\documentclass[10pt]{amsart}
\usepackage{amsmath}
\usepackage{amsthm}
\usepackage{amscd}
\usepackage{amsfonts}
\usepackage{amssymb}
\usepackage{graphicx}
\usepackage{color}
\usepackage[bookmarksnumbered, colorlinks, plainpages]{hyperref}
\usepackage{hyperref}
\hypersetup{colorlinks=true,linkcolor=red, anchorcolor=green, citecolor=cyan, urlcolor=red, filecolor=magenta, pdftoolbar=true}

\newtheorem{theorem}{Theorem}[section]

\newtheorem{proposition}[theorem]{Proposition}
\newtheorem{corollary}[theorem]{Corollary}
\theoremstyle{definition}
\newtheorem{definition}[theorem]{Definition}
\newtheorem{example}[theorem]{Example}

\newtheorem{remark}[theorem]{Remark}
\numberwithin{equation}{section}

\begin{document}
\setcounter{page}{1}

\title[Some results]{{Structural properties of }$2$\textbf{-}$C${-normal
operators}}

\author[Messaoud Guesba, Ismail Lakehal and
Sid Ahmed Ould Ahmed Mahmoud]{Messaoud Guesba$^{1}$, Ismail Lakehal$^{2}$ and  Sid Ahmed Ould Ahmed Mahmoud
$^{3}$}

\address{ $^{1,2}$ Department of Mathematics, Faculty of Exact Sciences,\\
El Oued University, 39000 Algeria\\
 }
\email{guesba-messaoud@univ-eloued.dz, guesbamessaoud2@gmail.com \ \ \\
ismaillakehal28@gmail.com}
\address{$^{c}$Mathematics Department, College of Science. Jouf University\\
Sakaka 2014. Saudi Arabia\\}
\email{sidahmed@ju.edu.sa}


\begin{abstract}

In this paper, we introduce and study a new class of bounded linear operators on complex Hilbert spaces, which we call \textit{2-\(C\)-normal operators}. This class is inspired by and closely related to the notion of 2-normal operators, with additional structure imposed via conjugation. Specifically, a bounded linear operator \( S \) on a complex Hilbert space \( H \) is said to be \textit{2-\(C\)-normal} if there exists a conjugation \( C \) on \( H \) such that
\[
(SC)^2 (SC)^\# = (SC)^\# (SC)^2.
\]
We investigate various fundamental properties of 2-\(C\)-normal operators, including algebraic characterizations, spectral properties, and examples that distinguish them from classical normal and 2-normal operators. Additionally, we explore how this class behaves under common operator transformations and examine its relation to other well-studied families of operators. The results presented contribute to a deeper understanding of conjugation-influenced operator behavior and open avenues for further study in the context of complex symmetric operator theory.

\end{abstract} \maketitle

\section{Introduction and preliminaries}

In the present work, we consider $\mathcal{K}$ a nontrivial separable
complex Hilbert space, equipped with the norm $\left\Vert .\right\Vert $
induced by the inner product $\left\langle .,.\right\rangle $. The set $%
\mathcal{B}(\mathcal{K})$ stands for the algebra of all bounded linear
operators acting on $\mathcal{K}$, and the identity operator is denoted by $%
I $.

An operator $S\in \mathcal{B}(\mathcal{K})$ is called normal if $S^{\ast
}S=SS^{\ast }$, $2$-normal if $S^{2}S^{\ast }=S^{\ast }S^{2}$, $n$-normal if
$S^{n}S^{\ast }=S^{\ast }S^{n}$ for $n\in
\mathbb{N}
$, isometry if $S^{\ast }S=I$, unitary if $S^{\ast }S=SS^{\ast }=I$, and $S$
is selfadjoint if $S^{\ast }=S$, where $S^{\ast }$ is the adjoint of $S$.
Further details concerning these classes can be found in \cite%
{ALZ,B,C,GU2,GU,JI,W}.

Let $S\in \mathcal{B}(\mathcal{K})$, We write $\ker \left( S\right) $ for
the kernel of $S$ and $\sigma _{ap}\left( S\right) $ for its approximate
point spectrum. A conjugation on the Hilbert space $\mathcal{K}$ is defined
as an antilinear mapping $C:\mathcal{K}\longrightarrow \mathcal{K}$ that
fulfills the following conditions:

(i) $C$ is antilinear, i.e., $C(\alpha x+y)=\overline{\alpha }Cx+Cy$ for $%
x,y\in \mathcal{K}$ and $\alpha \in
\mathbb{C}
$,

(ii) $C$ is invertible with $C^{-1}=C$ $(C^{2}=I)$,

(iii) $\left\langle Cx,Cy\right\rangle =\left\langle y,x\right\rangle $, for
all $x,y\in \mathcal{K}$.

We observe that if $S\in \mathcal{B}(\mathcal{K})$, then $CSC\in \mathcal{B}(%
\mathcal{K})$ and $\left( CSC\right) ^{\ast }=CS^{\ast }C$ . An operator $%
S\in \mathcal{B}(\mathcal{K})$ is called $C$-symmetric if there exists a
conjugation $C$ on $\mathcal{K}$ such that $CSC=S^{\ast }$ (see \cite%
{GAR,GAR1}). The study of $C$-symmetric was first developed by Garcia,
Putinar, and Wogen \cite{GAR1, GAR2}. This class includes, all normal
operators, Hankel operators, all $2\times 2$ matrices, algebraic operators
of order $2$, truncated Toeplitz operators, and the Volterra integration
operator (see \cite{J,Z,Z2}).

Given a conjugation $C$ \ on $\mathcal{K}$, an operator $S\in \mathcal{B}(%
\mathcal{K})$ is said to be $C$-normal if $CS^{\ast }SC=SS^{\ast }$, or
equivalently, $C|S|C=|S^{\ast }|$, where $|S|=(S^{\ast }S)^{\frac{1}{2}}$.
This notion was first introduced by Ptak et al. in 2020 by Ptak et al. \cite%
{PT}, and it contains the class of $C$-symmetric operators. It is also
immediate that $S$ is $C$-normal if and only if its adjoint is $C$-normal.

Wang et al. \cite{WA} investigated the structural aspects of $C$-normal
operators. Later, in 2021, Ko et al. \cite{KO} analyzed several properties
of this class. In particular, they showed that $S-\lambda I$ is $C$-normal
operator for all $\lambda \in
\mathbb{C}
$ if and only if $S$ is a complex symmetric. More recently, additional
contributions concerning $C$-normal operators have appeared in \cite{L}.

Every operator $S$ in $\mathcal{B}(\mathcal{K})$ admits a unique polar
decomposition of the form $S=U|S|$, where $U$ partial isometry determined by
the conditions $\ker \left( U\right) =\ker \left( \left\vert S\right\vert
\right) $ and $\ker \left( U^{\ast }\right) =\ker \left( S^{\ast }\right) $
(see \cite{KO}). Moreover, if $C$ is a conjugation on $\mathcal{K}$ and $U$
is unitary, then for all $n\in
\mathbb{N}
$ the operator $(U^{\ast })^{n}CU^{n}$ defines again a conjugation on $%
\mathcal{K}$ (see \cite{PT}).

Recall that if $X$ is a bounded antilinear operator on $\mathcal{H}$, then
there exists a unique bounded antilinear $X^{\#}$ on $\mathcal{H}$, denoted
the antilinear adjoint of $X$, which satisfies
\begin{equation}
\left\langle Xx,y\right\rangle =\overline{\left\langle
x,X^{\#}y\right\rangle }\text{, }\forall x,y\in \mathcal{K}\text{.}
\label{X}
\end{equation}

Clearly, $C^{\#}=C$ and for $S\in \mathcal{B}(\mathcal{K}$ then $%
(CS)^{\#}=S^{\ast }C$ and $(SC)^{\#}=CS^{\ast }$.

Moreover, for antilinear operators $X$ and $Y$, we get immediately from (\ref%
{X}) that $\left( X^{\#}\right) ^{\#}=X$ and $\left( X+Y\right)
^{\#}=X^{\#}+Y^{\#}$.

An antilinear operator $S$ is said to be antilinearly normal when it
satisfies $S^{\#}S=SS^{\#}$ \cite{PT}, and it is $C$-normal if $(SC)^{\#}$
commutes with $SC$ .

\bigskip This paper aims to present the structure of a particular class of
bounded linear operators, known as $2$-$C$-normal operators, and to discuss
several of their properties.
\section{ Main results}

Here, we introduce $2$-$C$-normal operators and explore some of their
characteristic properties. Firstly, we define $2$-$C$-normal operators as
follows. In all of the following, we consider $C$ a conjugation on $\mathcal{%
K}$.

\begin{definition}
An operator $S\in \mathcal{B}(\mathcal{K})$ is called $2$-$C$-normal whenever%
\begin{equation*}
(SC)^{2}(SC)^{\#}=(SC)^{\#}(SC)^{2}\text{.}
\end{equation*}
\end{definition}

\begin{remark}
(i) It is not difficult by this definition to verify that every $C$-normal
is $2$-$C$-normal .
\end{remark}

(ii) Note that, since $(SC)^{\#}=CS^{\ast }$, the operator $S\in \mathcal{B}(%
\mathcal{K})$ satisfies the $2$-$C$-normality condition precisely when%
\begin{equation*}
(SC)^{2}\left( CS^{\ast }\right) =\left( CS^{\ast }\right) (SC)^{2}\text{.}
\end{equation*}

(iii) It is clear that if $S$ is $2$-$C$-normal, the operator $\alpha S$ \
remains $2$-$C$-normal for all scalars $\alpha \in
\mathbb{C}
$.

The next examples illustrate the existence of an operator that is $2$-$C$%
-normal but fails to be $C$-normal.

\begin{example}
Take $S=\left(
\begin{array}{cc}
3 & 5 \\
1 & 3%
\end{array}%
\right) $, with the conjugation $C:%
\mathbb{C}
^{2}\rightarrow $ $%
\mathbb{C}
^{2}$, $C(z_{1},z_{2})=(\overline{z_{2}},\overline{z_{1}})$. One can check
by direct calculation that $S$ satisfies $2$-$C$-normality but does not
satisfy $C$-normality.
\end{example}

\begin{example}
Consider the canonical conjugation $C$ on a complex Hilbert space $\mathcal{H%
}$, given by
\begin{equation*}
C\bigg(\sum_{k=1}^{\infty }\omega _{k}e_{k}\bigg)=\sum_{k=1}^{\infty }%
\overline{\omega _{k}}e_{k}\text{,}
\end{equation*}%
where $(e_{k})_{k}$ is an orthonormal basis of $\mathcal{H}$ with $%
Ce_{k}=e_{k}$. For a complex sequence $\big(\omega _{k}\big)_{k\geq 1},$ the
operator $T$ on $\mathcal{H}$, defined as the weighted shift, is given by
\begin{equation*}
Te_{k}=\omega _{k}e_{k+1}\;\;\;\hbox{for all}\;\;k\geq 1.
\end{equation*}%
Recall that $T$ is bounded precisely when the sequence of weights $\big(%
\omega _{k}\big)_{k\geq 1}$ is bounded. Consider the case when $\big(\omega
_{k}\big)_{k\geq 1}$ is bounded and satisfies the conditions $|\omega
_{k_{0}}|\not=|\omega _{k_{0}-1}|$ for some integer $k_{0}$ and $|\omega
_{k+1}|=|\omega _{k-1}|$ for $k\geq 2.$ According to \cite[Proposition 3.1]%
{LLL} establishes that $T$ fails to be $C$-normal. While a straightforward
computation indicates that
\begin{equation*}
\bigg(CT^{\ast }\big(CT\big)^{2}-\big(CT\big)^{2}CT^{\ast }\bigg)e_{k}=%
\overline{\omega _{k}}\bigg(|w_{k+1}|^{2}-|\omega _{k-1}|^{2}\bigg)e_{k+1}=0%
\text{.}
\end{equation*}

Therefore, we conclude that $T$ is $2$-$C$-normal.
\end{example}

\begin{example}
Let us consider the Hilbert space $L^{2}\big(\lbrack 0,2\pi ],\mathbb{C}\big)
$ and define the conjugation operator $C$ on $L^{2}\big(\lbrack 0,2\pi ],%
\mathbb{C}\big)$ given by $(Ch)(\theta )=\overline{h(2\pi -\theta )}$ for $%
\theta \in \lbrack 0,2\pi ].$ Assume that $\psi \in L^{\infty }\big(\lbrack
0,2\pi ],\mathbb{C}\big)$. Let $M_{\psi }\in \mathcal{B}\big(L^{2}\big(%
\lbrack 0,2\pi ],\mathbb{C}\big)\big)$ be the multiplication operator given
by $M_{\psi }h=\psi h$. Then $M_{\psi }$ is $2$-$C$-normal operator.

\vskip0.2 cm \noindent In fact, a direct calculation shows that
\begin{equation*}
CM_{\psi }^{\ast }\big(CM_{\psi }\big)^{2}h(\theta )=\psi (\theta )|\psi
(2\pi -\theta )|^{2}\overline{h(2\pi -\theta )},
\end{equation*}%
and
\begin{equation*}
\big(CM_{\psi }\big)^{2}CM_{\psi }^{\ast }h(\theta )=\psi (\theta )|\psi
(2\pi -\theta )|^{2}\overline{h(2\pi -\theta )}.
\end{equation*}%
Therefore, $M_{\psi }$ is $2$-$C$-normal. However $M_{\psi }$ fails to be
necessarily $C$-normal by \cite[Example 2]{KO}.
\end{example}

\begin{proposition}
If $S\in \mathcal{B}(\mathcal{K})$ is $2$-$C$-normal, then $\left( SC\right)
^{2}$ is normal operator.
\end{proposition}

\begin{proof}
For every $x,y\in \mathcal{K}$, we obtain
\begin{eqnarray*}
\left\langle (SC)^{\#}(SC)^{2}x,y\right\rangle &=&\overline{\left\langle
(SC)^{2}x,(SC)y\right\rangle } \\
&=&\overline{\left\langle x,((SC)^{2})^{\ast }(SC)y\right\rangle }\text{,}
\end{eqnarray*}

and $\ \ \ \ \ \ \ \ \ \ \ \ \ \ \ \ \ \ \ \ \ \ \ \ \ \ \ $%
\begin{eqnarray*}
\left\langle (SC)^{2}(SC)^{\#}x,y\right\rangle &=&\left\langle
(SC)^{\#}x,((SC)^{2})^{\ast }y\right\rangle \\
&=&\overline{\left\langle x,(SC)((SC)^{2})^{\ast }y\right\rangle }\text{.}
\end{eqnarray*}

This implies that
\begin{equation*}
((SC)^{2})^{\ast }(SC)=(SC)((SC)^{2})^{\ast }\text{.}
\end{equation*}

Therefore, we infer that%
\begin{equation*}
((SC)^{2})^{\ast }(SC)(SC)=(SC)((SC)^{2})^{\ast }(SC)\text{.}
\end{equation*}

Hence
\begin{eqnarray*}
((SC)^{2})^{\ast }(SC)^{2} &=&(SC)((SC)^{2})^{\ast }(SC) \\
&=&(SC)^{2}((SC)^{2})^{\ast }\text{.}
\end{eqnarray*}

This means that $(SC)^{2}$ is a normal.
\end{proof}

\begin{definition}
Let $X$ be an antilinearly operator. We said that $X$ is antilinearly $2$%
-normal if
\begin{equation*}
X^{2}X^{\#}=X^{\#}X^{2}\text{.}
\end{equation*}
\end{definition}

\begin{proposition}
Every antilinear normal operator $X$ is necessarily antilinear $2$-normal.
\end{proposition}

\begin{proof}
Since $X$ is antilinearly normal, then
\begin{equation*}
XX^{\#}=X^{\#}X\text{,}
\end{equation*}%
it follows that%
\begin{equation*}
X^{2}X^{\#}=XXX^{\#}=XX^{\#}X=X^{\#}X^{2}\text{.}
\end{equation*}

Then, $X$ is antilinearly $2$-normal.
\end{proof}

\begin{theorem}
\label{TH1}For $S\in \mathcal{B}(\mathcal{K})$, the operator $S$ is $2$-$C$%
-normal operator if and only if the relation
\begin{equation*}
CSS^{\ast }CS^{\ast }=S^{\ast }CS^{\ast }SC
\end{equation*}
\end{theorem}

is satisfied.

\begin{proof}
Because $S$ belongs to the class of $C$-normal operators, then for any $%
x,y\in \mathcal{K}$, we can observe that%
\begin{eqnarray*}
\left\langle (SC)^{2}(SC)^{\#}x,y\right\rangle &=&\left\langle
SCSC(SC)^{\#}x,y\right\rangle \\
&=&\left\langle CSC(SC)^{\#}x,S^{\ast }y\right\rangle \\
\ &=&\left\langle CS^{\ast }y,SC(SC)^{\#}x\right\rangle \\
&=&\left\langle S^{\ast }CS^{\ast }y,C(SC)^{\#}x\right\rangle \\
&=&\left\langle (SC)^{\#}x,CS^{\ast }CS^{\ast }y\right\rangle \\
&=&\overline{\left\langle x,(SC)CS^{\ast }CS^{\ast }y\right\rangle } \\
\ &=&\overline{\left\langle x,(SC)(CS^{\ast })^{2}y\right\rangle }\text{.}
\end{eqnarray*}%
On the other hand, it holds that
\begin{eqnarray*}
&&\left\langle (SC)^{\#}(SC)^{2}x,y\right\rangle \\
&=&\overline{\left\langle SCSCx,(SC)y\right\rangle } \\
&=&\overline{\left\langle CSCx,S^{\ast }(SC)y\right\rangle } \\
\ &=&\overline{\left\langle CS^{\ast }(SC)y,SCx\right\rangle } \\
&=&\overline{\left\langle S^{\ast }CS^{\ast }(SC)y,Cx\right\rangle } \\
\ &=&\overline{\left\langle x,CS^{\ast }CS^{\ast }(SC)y\right\rangle }\  \\
&=&\overline{\left\langle x,(CS^{\ast })^{2}(SC)y\right\rangle }\text{.}
\end{eqnarray*}%
Therefore, we get%
\begin{equation*}
(SC)(CS^{\ast })^{2}=(CS^{\ast })^{2}(SC)\text{.}
\end{equation*}

Thus,
\begin{equation*}
SCCS^{\ast }CS^{\ast }=CS^{\ast }CS^{\ast }SC\text{.}
\end{equation*}%
Hence,
\begin{equation*}
CSS^{\ast }CS^{\ast }=S^{\ast }CS^{\ast }SC\text{.}
\end{equation*}

This completes the proof.
\end{proof}

\begin{proposition}
Let $\{e_{i}\}_{i=1}^{2n+1}$ be an orthonormal basis of $\
\mathbb{C}
^{2n+1}$ and define the operator $S$ as follows
\begin{equation*}
Se_{i}=\alpha _{i}e_{i+1}\text{ and }Se_{2n+1}=0\text{,}
\end{equation*}
\end{proposition}

where $(\alpha _{i})_{i=1}^{2n}\in
\mathbb{C}
$. Let $C$ be a conjugation on $%
\mathbb{C}
^{2n+1}$ defined by
\begin{equation*}
C(z_{1},z_{2},....,z_{2n},z_{2n+1})=(\overline{z_{2n+1}},\overline{z_{2n}}%
,....,\overline{z_{2}},\overline{z_{1}})\text{.}
\end{equation*}%
Then, $S$ is $2$-$C$-normal operator for all $(\alpha _{i})_{i=1}^{2n}\in
\mathbb{C}
$.

\begin{proof}
For every $z\in
\mathbb{C}
^{2n+1}$, we have
\begin{eqnarray*}
S^{\ast }CS^{\ast }SC(z) &=&S^{\ast }C\left(
\begin{array}{ccccccc}
\left\vert \alpha _{1}\right\vert ^{2} & 0 & 0 & .. & .. & .. & 0 \\
0 & \left\vert \alpha _{2}\right\vert ^{2} & 0 & .. & .. & .. & .. \\
0 & 0 & \left\vert \alpha _{3}\right\vert ^{2} & .. & .. & .. & .. \\
.. & .. & .. & .. & .. & .. & .. \\
.. & .. & .. & .. & .. & .. & .. \\
.. & .. & .. & .. & .. & \left\vert \alpha _{2n}\right\vert ^{2} & 0 \\
0 & .. & .. & .. & .. & 0 & 0%
\end{array}%
\right) \left(
\begin{array}{c}
\overline{z_{2n+1}} \\
\overline{z_{2n}} \\
\overline{z_{2n-1}} \\
.. \\
.. \\
\overline{z_{2}} \\
\overline{z_{1}}%
\end{array}%
\right) \\
&=&\left(
\begin{array}{ccccccc}
0 & \overline{\alpha _{1}} & 0 & 0 & .. & .. & 0 \\
0 & 0 & \overline{\alpha _{2}} & 0 & .. & .. & .. \\
0 & 0 & 0 & \overline{\alpha _{3}} & .. & .. & .. \\
.. & .. & .. & .. & .. & .. & .. \\
.. & .. & .. & .. & .. & \overline{\alpha _{2n-1}} & 0 \\
.. & .. & .. & .. & .. & 0 & \overline{\alpha _{2n}} \\
0 & .. & .. & .. & .. & 0 & 0%
\end{array}%
\right) \left(
\begin{array}{c}
0 \\
\left\vert \alpha _{2n}\right\vert ^{2}z_{2} \\
\left\vert \alpha _{2n-1}\right\vert ^{2}z_{3} \\
.. \\
\left\vert \alpha _{3}\right\vert ^{2}z_{2n-1} \\
\left\vert \alpha _{2}\right\vert ^{2}z_{2n} \\
\left\vert \alpha _{1}\right\vert ^{2}z_{2n+1}%
\end{array}%
\right) \\
&=&\left(
\begin{array}{c}
\overline{\alpha _{1}}\left\vert \alpha _{2n}\right\vert ^{2}z_{2} \\
\overline{\alpha _{2}}\left\vert \alpha _{2n-1}\right\vert ^{2}z_{3} \\
\overline{\alpha _{3}}\left\vert \alpha _{2n-2}\right\vert ^{2}z_{4} \\
.. \\
\overline{\alpha _{2n-1}}\left\vert \alpha _{2}\right\vert ^{2}z_{2n} \\
\overline{\alpha _{2n}}\left\vert \alpha _{1}\right\vert ^{2}z_{2n+1} \\
0%
\end{array}%
\right) \text{.}
\end{eqnarray*}%
In a similar way, we can see that%
\begin{eqnarray*}
CSS^{\ast }CS^{\ast }(z) &=&CSS^{\ast }C\left(
\begin{array}{ccccccc}
0 & \overline{\alpha _{1}} & 0 & 0 & .. & .. & 0 \\
0 & 0 & \overline{\alpha _{2}} & 0 & .. & .. & .. \\
0 & 0 & 0 & \overline{\alpha _{3}} & .. & .. & .. \\
.. & .. & .. & .. & .. & .. & .. \\
.. & .. & .. & .. & .. & \overline{\alpha _{2n-1}} & 0 \\
.. & .. & .. & .. & .. & 0 & \overline{\alpha _{2n}} \\
0 & .. & .. & .. & .. & 0 & 0%
\end{array}%
\right) \left(
\begin{array}{c}
z_{1} \\
z_{2} \\
z_{3} \\
.. \\
.. \\
z_{2n} \\
z_{2n+1}%
\end{array}%
\right) \\
&=&C\left(
\begin{array}{ccccccc}
0 & 0 & 0 & .. & .. & .. & 0 \\
0 & \left\vert \alpha _{1}\right\vert ^{2} & 0 & .. & .. & .. & .. \\
0 & 0 & \left\vert \alpha _{2}\right\vert ^{2} & .. & .. & .. & .. \\
.. & .. & .. & .. & .. & .. & .. \\
.. & .. & .. & .. & .. & .. & .. \\
.. & .. & .. & .. & .. & \left\vert \alpha _{2n-1}\right\vert ^{2} & 0 \\
0 & .. & .. & .. & .. & 0 & \left\vert \alpha _{2n}\right\vert ^{2}%
\end{array}%
\right) \left(
\begin{array}{c}
0 \\
\alpha _{2n}\overline{z_{2n+1}} \\
\alpha _{2n-1}\overline{z_{2n}} \\
.. \\
\alpha _{3}\overline{z_{4}} \\
\alpha _{2}\overline{z_{3}} \\
\alpha _{1}\overline{z_{2}}%
\end{array}%
\right)
\end{eqnarray*}%
\begin{equation*}
=C\left(
\begin{array}{c}
0 \\
\left\vert \alpha _{1}\right\vert ^{2}\alpha _{2n}\overline{z_{2n+1}} \\
\left\vert \alpha _{2}\right\vert ^{2}\alpha _{2n-1}\overline{z_{2n}} \\
.. \\
\left\vert \alpha _{2n-1}\right\vert ^{2}\alpha _{3}\overline{z_{4}} \\
\left\vert \alpha _{2n-1}\right\vert ^{2}\alpha _{2}\overline{z_{3}} \\
\left\vert \alpha _{2n}\right\vert ^{2}\alpha _{1}\overline{z_{2}}%
\end{array}%
\right) =\left(
\begin{array}{c}
\overline{\alpha _{1}}\left\vert \alpha _{2n}\right\vert ^{2}z_{2} \\
\overline{\alpha _{2}}\left\vert \alpha _{2n-1}\right\vert ^{2}z_{3} \\
\overline{\alpha _{3}}\left\vert \alpha _{2n-1}\right\vert ^{2}z_{4} \\
.. \\
\overline{\alpha _{2n-1}}\left\vert \alpha _{2}\right\vert ^{2}z_{2n} \\
\overline{\alpha _{2n}}\left\vert \alpha _{1}\right\vert ^{2}z_{2n+1} \\
0%
\end{array}%
\right) \text{.}
\end{equation*}

This shows that%
\begin{equation*}
S^{\ast }CS^{\ast }SC=CSS^{\ast }CS^{\ast }\text{ for all }(\alpha
_{i})_{i=1}^{2n}\in
\mathbb{C}
\text{.}
\end{equation*}%
Hence, it can be deduced that $S$ is $2$-$C$-normal.
\end{proof}

\begin{proposition}
If $S\in \mathcal{B}(\mathcal{K})$ is both self-adjoint and $2$-$C$-normal,
then $S$ is commutative with $CS^{2}C\ $.$\ \ \ \ \ \ \ \ \ \ $
\end{proposition}

\begin{proof}
Assuming $S$ is $2$-$C$-normal, then applying Theorem \ref{TH1}, we obtain
\begin{equation*}
CSS^{\ast }CS^{\ast }\ =S^{\ast }CS^{\ast }SC\text{.}\ \ \ \ \
\end{equation*}

This implies that
\begin{equation*}
CS^{2}CS\ =SCS^{2}C.
\end{equation*}

Thus, $S$ is commutative with $CS^{2}C$ .
\end{proof}

\begin{example}
Let $S=\left(
\begin{array}{cc}
2 & 1 \\
1 & 3%
\end{array}%
\right) $, and let $C$ be a conjugation given by\ $C(z_{1},z_{2})=(\overline{%
z_{1}},\overline{z_{2}})$. We observe that $S$ is $2$-$C$-normal and $%
S^{\ast }=S$, then
\begin{equation*}
SCS^{2}C=CS^{2}CS\text{.}
\end{equation*}%
$\ $
\end{example}

\begin{theorem}
\label{TH2} Assuming $S$ is $2$-$C$-normal, the statement below is valid.
\end{theorem}

If $S$ is an invertible, then $S^{-1}$ is a $2$-$C$-normal operator.

\begin{proof}
Suppose that $S$ is an invertible, then%
\begin{eqnarray*}
&&CS^{-1}(S^{-1})^{\ast }C(S^{-1})^{\ast } \\
&=&CS^{-1}(S^{\ast })^{-1}C(S^{\ast })^{-1} \\
&=&C(S^{\ast }S)^{-1}C(S^{\ast })^{-1} \\
\ &=&(CS^{\ast }SC)^{-1}(S^{\ast })^{-1} \\
&=&(S^{\ast }CS^{\ast }SC)^{-1} \\
&=&(CS^{\ast }SCS^{\ast })^{-1} \\
&=&(S^{-1})^{\ast }C(S^{-1})^{\ast }S^{-1}C\text{.}
\end{eqnarray*}%
\ \ \ \ \ \ \ \ \ \ \ \ \ \ \ \ \ \ \ \ \ \ \ \

Thus, $S^{-1}$ turns out to be $2$-$C$-normal.
\end{proof}

\begin{proposition}
Let $U\in \mathcal{B}(\mathcal{K})$ be unitary operator. Thus, the next
statements are valid

$(i)$ $U$ is a $2$-$C$-normal operator.

$(ii)$ $U^{n}$ is $2$-$C$-normal for all $n\in
\mathbb{N}
$.
\end{proposition}

\begin{proof}
$(i)$ As$\ U$ is unitary, it satisfies $U^{\ast }U=UU^{\ast }=I$.\ Hence%
\begin{equation*}
CUU^{\ast }CU^{\ast }=CCU^{\ast }=U^{\ast }\text{,}
\end{equation*}

and%
\begin{equation*}
U^{\ast }CU^{\ast }UC=U^{\ast }CC=U^{\ast }\text{.}
\end{equation*}

Thus,%
\begin{equation*}
CUU^{\ast }CU^{\ast }=U^{\ast }CU^{\ast }UC\text{.}
\end{equation*}

Hence, $U$ is a $2$-$C$-normal operator.

$(ii)$ Notice that $U^{n}(U^{n})^{\ast }=U^{n}(U^{\ast })^{n}=UU...%
\underbrace{UU^{\ast }}...U^{\ast }U^{\ast }=I$. A similar method in the
proof $(i)$, we can obtain the desired result.
\end{proof}

\begin{theorem}
\label{TH3}Let $U\in \mathcal{B}(\mathcal{K})$ be unitary. Then, $S\in
\mathcal{B}(\mathcal{K})$ is a $2$-$C_{1}$-normal if and only if $USU^{\ast
} $ is $2$-$C$-normal such that $C_{1}=U^{\ast }CU$.
\end{theorem}

\begin{proof}
It is easy to see that $C_{1}=U^{\ast }CU$ is a conjugation on $\mathcal{K}$%
. If $USU^{\ast }$ is a $2$-$C$-normal, then
\begin{eqnarray*}
&&C(USU^{\ast })(USU^{\ast })^{\ast }C(USU^{\ast })^{\ast } \\
&=&(USU^{\ast })^{\ast }C(USU^{\ast })^{\ast }(USU^{\ast })C \\
&\Leftrightarrow &CUSS^{\ast }U^{\ast }CUS^{\ast }U^{\ast }=US^{\ast
}U^{\ast }CUS^{\ast }SU^{\ast }C \\
&\Leftrightarrow &UU^{\ast }CUSS^{\ast }U^{\ast }CUS^{\ast }U^{\ast
}=US^{\ast }U^{\ast }CUS^{\ast }SU^{\ast }CUU^{\ast } \\
&\Leftrightarrow &UC_{1}SS^{\ast }C_{1}S^{\ast }U^{\ast }=US^{\ast
}C_{1}S^{\ast }SC_{1}U^{\ast } \\
&\Leftrightarrow &C_{1}SS^{\ast }C_{1}S^{\ast }=S^{\ast }C_{1}S^{\ast }SC_{1}%
\text{.}
\end{eqnarray*}

So, $S$ \ belongs to the class of $2$-$C_{1}$-normal operators exactly when $%
USU^{\ast }$ is $2$-$C$-normal.\ \ \ \ \ \ \ \ \ \ \ \ \ \ \ \ \ \ \ \ \ \ \
\ \ \ \ \ \ \ \ \ \ \ \ \ \ \ \ \ \ \ \ \ \ \ \ \ \ \ \ \ \ \ \ \ \ \ \ \ \
\ \ \ \ \ \ \ \ \ \ \ \ \ \ \ \ \ \ \ \ \ \ \ \ \ \
\end{proof}

\begin{proposition}
Let $S\in \mathcal{B}(\mathcal{K})$, and let $U\in \mathcal{B}(\mathcal{K})$
be a unitary. Thus, the next assertions are valid

$(i)$ $U^{n}S(U^{\ast })^{n}$ is $2$-$C$-normal is equivalent to $S$ is $2$-$%
(U^{\ast })^{n}CU^{n}$-normal, for all $n\in
\mathbb{N}
$.

$(ii)$ $U^{n}S^{m}(U^{\ast })^{n}$ is $2$-$C$-normal is equivalent to$\
S^{m} $ is $2$-$(U^{\ast })^{n}CU^{n}$-normal for all $n,m\in
\mathbb{N}
$.
\end{proposition}

\begin{proof}
The argument follows the same line as in Theorem \ref{TH3}.
\end{proof}

\begin{proposition}
Let $C(z_{1},z_{2},z_{3})=(\overline{z_{3}},\overline{z_{2}},\overline{z_{1}}%
)$, and $a,b,$ $c$ be real numbers and $S=\left(
\begin{array}{ccc}
0 & a & b \\
0 & 0 & c \\
0 & 0 & 0%
\end{array}%
\right) $.Then, $T$ is $2$-$C$-normal when $b=0$ and $a,c$ arbitrary or $%
b\neq 0$ and $a=c$.
\end{proposition}

\begin{proof}
Let $z=\left(
\begin{array}{c}
z_{1} \\
z_{2} \\
z_{3}%
\end{array}%
\right) \in
\mathbb{C}
^{3}$, we have%
\begin{eqnarray*}
CSS^{\ast }CS^{\ast }(z) &=&CSS^{\ast }C\left(
\begin{array}{ccc}
0 & 0 & 0 \\
a & 0 & 0 \\
b & c & 0%
\end{array}%
\right) \left(
\begin{array}{c}
z_{1} \\
z_{2} \\
z_{3}%
\end{array}%
\right) \\
\ &=&C\left(
\begin{array}{ccc}
0 & a & b \\
0 & 0 & c \\
0 & 0 & 0%
\end{array}%
\right) \left(
\begin{array}{ccc}
0 & 0 & 0 \\
a & 0 & 0 \\
b & c & 0%
\end{array}%
\right) \left(
\begin{array}{c}
b\overline{z_{1}}+c\overline{z_{2}} \\
a\overline{z_{1}} \\
0%
\end{array}%
\right) \\
&=&C\left(
\begin{array}{c}
(ba^{2}+b^{3}+abc)\overline{z_{1}}+ca^{2}+cb^{2}\overline{z_{2}} \\
(ac^{2}+b^{2}c)\overline{z_{1}}+bc^{2}\overline{z_{2}} \\
0%
\end{array}%
\right) \allowbreak \\
&=&\left(
\begin{array}{c}
0 \\
(ac^{2}+b^{2}c)z_{1}+bc^{2}z_{2} \\
(ba^{2}+b^{3}+abc)z_{1}+(ca^{2}+cb^{2})z_{2}%
\end{array}%
\right) \text{.}
\end{eqnarray*}

From another point of view, we obtain
\begin{eqnarray*}
S^{\ast }CS^{\ast }SC(z) &=&S^{\ast }CS^{\ast }\left(
\begin{array}{ccc}
0 & a & b \\
0 & 0 & c \\
0 & 0 & 0%
\end{array}%
\right) \left(
\begin{array}{c}
\overline{z_{3}} \\
\overline{z_{2}} \\
\overline{z_{1}}%
\end{array}%
\right) \\
\ &=&S^{\ast }C\left(
\begin{array}{c}
0 \\
ab\overline{z_{1}}+a^{2}\overline{z_{2}} \\
(b^{2}+c^{2})\overline{z_{1}}+ba\overline{z_{2}}%
\end{array}%
\right) \\
\ &=&\allowbreak \left(
\begin{array}{c}
0 \\
\left( ab^{2}+ac^{2}\right) z_{1}+a^{2}bz_{2} \\
\left( abc+b^{3}+bc^{2}\right) z_{1}+(ab^{2}+ca^{2})z_{2}%
\end{array}%
\right) \allowbreak \ \text{.}
\end{eqnarray*}

Hence, $S$ becomes a $2$-$C$-normal whenever \
\begin{equation*}
\left\{
\begin{array}{c}
ac^{2}+b^{2}c=ab^{2}+ac^{2}\text{ } \\
\text{ }bc^{2}=a^{2}b \\
ba^{2}+b^{3}+abc=abc+b^{3}+bc^{2}\text{ } \\
\text{ }ca^{2}+cb^{2}=ab^{2}+ca^{2}%
\end{array}%
\right.
\end{equation*}%
then%
\begin{equation*}
\left\{
\begin{array}{c}
b^{2}(a-c)=0\text{ } \\
b(a^{2}-c^{2})=0%
\end{array}%
\right.
\end{equation*}%
equivalent if $b=0$ and $a,c$ arbitrary or $b\neq 0$ and $a=c$, then $S$ \
is $2$-$C$-normal in these cases. This means that $T$ takes one of the
following two forms:
\begin{equation*}
S=\left(
\begin{array}{ccc}
0 & a & b \\
0 & 0 & a \\
0 & 0 & 0%
\end{array}%
\right) \text{ or }S=\left(
\begin{array}{ccc}
0 & a & 0 \\
0 & 0 & b \\
0 & 0 & 0%
\end{array}%
\right) \text{,}
\end{equation*}%
where $a$ and $b$ are real numbers arbitrary.
\end{proof}

\begin{theorem}
\label{TH4} Let $S\in \mathcal{B}(\mathcal{K})$ such that $CSC=S$. Then, $S$
being $2$-$C$-normal$\ $is equivalent to $S$ is $2$-normal.
\end{theorem}

\begin{proof}
Suppose that $CSC=S$, i.e., $CS=SC$, as a consequence $\ CS^{\ast }=S^{\ast
}C$.\ Also, we can observe that%
\begin{equation*}
CSS^{\ast }CS^{\ast }=SCS^{\ast }CS^{\ast }=SS^{\ast }CCS^{\ast }=S(S^{\ast
})^{2}
\end{equation*}

and%
\begin{equation*}
S^{\ast }CS^{\ast }SC=S^{\ast }S^{\ast }CSC=S^{\ast }S^{\ast }SCC=(S^{\ast
})^{2}S\text{.}
\end{equation*}

If $S$ is a $2$-$C$-normal operator, then
\begin{equation*}
CST^{\ast }CS^{\ast }=S^{\ast }CS^{\ast }SC\text{.}
\end{equation*}

This implies that
\begin{equation*}
S(S^{\ast })^{2}=(S^{\ast })^{2}S\text{.}
\end{equation*}

Therefore, $S^{2}S^{\ast }=S^{\ast }S^{2}$. So, $S$ is a $2$-$C$-normal$\ $%
is equivalent to$S$ is $2$-normal operator.
\end{proof}

\begin{remark}
Consider $(\alpha _{i})_{i=1}^{k}\in
\mathbb{C}
$, where $k\in
\mathbb{N}
$, let $S$ be defined as
\begin{equation*}
S=\left(
\begin{array}{ccccccccc}
0 & \overline{\alpha _{1}} & 0 & .. & .. & .. & .. & .. & 0 \\
\alpha _{1} & 0 & \overline{\alpha _{2}} & 0 & .. & .. & .. & .. & .. \\
0 & \alpha _{2} & .. & .. & .. & .. & .. & .. & .. \\
.. & .. & .. & 0 & \overline{\alpha _{k}} & 0 & .. & .. & .. \\
.. & .. & .. & \alpha _{k} & 0 & \overline{\alpha _{k}} & .. & .. & .. \\
.. & .. & .. & .. & \alpha _{k} & 0 & .. & .. & .. \\
.. & .. & .. & .. & .. & .. & .. & \overline{\alpha _{2}} & 0 \\
.. & .. & .. & .. & .. & .. & \alpha _{2} & 0 & \overline{\alpha _{1}} \\
0 & .. & .. & .. & .. & .. & .. & \alpha _{1} & 0%
\end{array}%
\right) \text{.}
\end{equation*}
\end{remark}

Take $C$ to be a conjugation on $%
\mathbb{C}
^{2k+1}$given as
\begin{equation*}
C(z_{1},z_{2},...,z_{2k+1})=(\overline{z_{2k+1}},...,\overline{z_{2}},%
\overline{z_{1}})\text{.}
\end{equation*}

By direct computation we can show that $S=S^{\ast }$, so $S$ is $2$-normal
and $CSC=S$, then by using Theorem \ref{TH4}, we get $S$ is $2$-$C$-normal.

\begin{theorem}
\label{TH5}If $S$ $\in \mathcal{B}(\mathcal{K})$ is invertible with $S^{\ast
}=S^{-1}$, then
\end{theorem}

$(i)$ $S\ $ belongs to the class of $2$-$C$-normal operators.

$(ii)$ $S\ $ satisfies $C$-normality.

\begin{proof}
$(i)$ Since $S^{\ast }=S^{-1}$, then%
\begin{equation*}
CSS^{\ast }CS^{\ast }=\ CSS^{-1}CS^{\ast }=CCS^{\ast }=S^{\ast }
\end{equation*}

and
\begin{equation*}
S^{\ast }CS^{\ast }SC=\ S^{\ast }CS^{-1}SC=S^{\ast }CC=S^{\ast }\text{.}
\end{equation*}%
\

This implies that
\begin{equation*}
CSS^{\ast }CS^{\ast }=S^{\ast }CS^{\ast }SC\text{.}
\end{equation*}%
\ Therefore, $S\ $is a $2$-$C$-normal operator.

$(ii)$ Obvious.
\end{proof}

\begin{proposition}
\label{Pro1}Let $S\in \mathcal{B}(\mathcal{K})$ be an invertible, and $%
S^{\ast }=\lambda S^{-1}$ for some $\lambda \in
\mathbb{C}
$. Then

$(i)$ $S\ $belongs to the class of $\ 2$-$C$-normal operators.

$(ii)$ $S$ satisfies $C$-normality if and only if $\ \lambda =\overline{%
\lambda }$, i.e., $\lambda \in
\mathbb{R}
$.
\end{proposition}

\begin{proof}
$(i)$ Since$\ S^{\ast }=\lambda S^{-1}$, then
\begin{equation*}
CSS^{\ast }CS^{\ast }=CS(\lambda S^{-1})CS^{\ast }=\overline{\lambda }%
CSS^{-1}CS^{\ast }=\overline{\lambda }S^{\ast }\text{.}
\end{equation*}%
On the other hand, we have%
\begin{equation*}
S^{\ast }CS^{\ast }SC=S^{\ast }C(\lambda S^{-1})SC=\overline{\lambda }%
S^{\ast }CS^{-1}SC=\overline{\lambda }S^{\ast }\text{.}
\end{equation*}

Thus,
\begin{equation*}
\ CSS^{\ast }CS^{\ast }=S^{\ast }CS^{\ast }SC\text{.}
\end{equation*}%
$\ $\ Then, $S$ is $2$-$C$-normal operator for $\lambda \in
\mathbb{C}
$.

$(ii)$ If $S^{\ast }=\lambda S^{-1}$, then we have
\begin{equation*}
SS^{\ast }=S(\lambda S^{-1})=\lambda SS^{-1}=\lambda \text{, }
\end{equation*}%
and
\begin{equation*}
CS^{\ast }SC=C(\lambda S^{-1})SC=\overline{\lambda }CS^{-1}SC=\overline{%
\lambda }CC=\overline{\lambda }\text{.}
\end{equation*}

\ This implies that $\ $%
\begin{equation*}
SS^{\ast }=CS^{\ast }SC\text{.}
\end{equation*}%
$\ $\ Therefore, $S$ \ belongs to the class $C$-normal if and only if $%
\lambda =\overline{\lambda }$.
\end{proof}

\begin{proposition}
\ Let $S\in \mathcal{B}(\mathcal{K})$ be with the polar decomposition $%
S=U\left\vert S\right\vert $, where $U$ is a unitary. Suppose $C$ and $J$
are conjugations on $\mathcal{K}$ such that $U=CJ$ . Then

$(i)$ $S\ $ is a $C$-normal if and only if \ $J\left\vert S\right\vert
^{2}J=\left\vert S\right\vert ^{2}$.

$(ii)$ $S\ $ is a $2$-$C$-normal if and only if \ $(J\left\vert S\right\vert
J)^{2}\left\vert S\right\vert =\left\vert S\right\vert (J\left\vert
S\right\vert J)^{2}$.
\end{proposition}

\begin{proof}
$(i)$ Take$S$ to be $C$-normal, it follows that
\begin{eqnarray*}
CSS^{\ast }C &=&CU\left\vert S\right\vert ^{2}U^{\ast }C \\
&=&CCJ\left\vert S\right\vert ^{2}JCC \\
&=&J\left\vert S\right\vert ^{2}J\ \text{.}
\end{eqnarray*}

Then, $S^{\ast }S=CSS^{\ast }C$ \ if and only if \ $J\left\vert S\right\vert
^{2}J=\left\vert S\right\vert ^{2}$. So, $S$ is a $C$-normal equivalent to
satisfying $J\left\vert S\right\vert ^{2}J=\left\vert S\right\vert ^{2}$.

$(ii)$ We have
\begin{eqnarray*}
CSS^{\ast }CS^{\ast } &=&CCJ\left\vert S\right\vert ^{2}JCC\left\vert
S\right\vert U^{\ast } \\
\ &=&J\left\vert S\right\vert ^{2}J\left\vert S\right\vert U^{\ast } \\
&=&(J\left\vert S\right\vert J)^{2}\left\vert S\right\vert JC
\end{eqnarray*}%
and
\begin{eqnarray*}
S^{\ast }CS^{\ast }SC &=&\left\vert S\right\vert U^{\ast }C\left\vert
S\right\vert U^{\ast }U\left\vert S\right\vert C \\
&=&\left\vert S\right\vert JCC\left\vert S\right\vert ^{2}C \\
&=&\left\vert S\right\vert J\left\vert S\right\vert ^{2}C\text{.}
\end{eqnarray*}

This implies that
\begin{equation*}
CSS^{\ast }CS^{\ast }=S^{\ast }CS^{\ast }SC\ \Leftrightarrow (J\left\vert
S\right\vert J)^{2}\left\vert S\right\vert JC=\left\vert S\right\vert
J\left\vert S\right\vert ^{2}C\text{.}
\end{equation*}%
Then,%
\begin{eqnarray*}
&&(J\left\vert S\right\vert J)^{2}\left\vert S\right\vert JCCJ \\
&=&\left\vert S\right\vert J\left\vert S\right\vert ^{2}CCJ \\
&\Leftrightarrow &(J\left\vert S\right\vert J)^{2}\left\vert S\right\vert
=\left\vert S\right\vert J\left\vert S\right\vert ^{2}J \\
&\Leftrightarrow &(J\left\vert S\right\vert J)^{2}\left\vert S\right\vert
=\left\vert S\right\vert (J\left\vert S\right\vert J)^{2}\text{.}
\end{eqnarray*}

Consequently, $S$ is $2$-$C$-normal exactly in the case where $(J\left\vert
S\right\vert J)^{2}\left\vert S\right\vert =\left\vert S\right\vert
(J\left\vert S\right\vert J)^{2}$.
\end{proof}

\begin{remark}
\ Let $S=U\left\vert S\right\vert $ be the polar decomposition of $S$, with $%
U$ a unitary, then $U=CJ$, with $C$ and $J$ be conjugations on $\mathcal{K}$%
. Moreover, if $\ \left\vert S\right\vert J=J\left\vert S\right\vert $, so $%
S\ $is both $C$-normal and $2$-$C$-normal.
\end{remark}

\begin{proposition}
Consider $(X,\mu )$ a measure space, and $f\in L^{2}(X,\mu )$. Let $\varphi
\in L^{\infty }$, Suppose $T_{\varphi }$ is a multiplication operator on $%
L^{2}(X,\mu )$, $S_{\varphi }f=\varphi f$ . As a consequence
\end{proposition}

$(i)$ If $Cf(x)=\overline{f(x)}$, then $S_{\varphi }$ belongs to the classes
$C$-normal and $2$-$C$-normal.

$(ii)$ If $Cf(x)=\overline{f(-x)}$, then $S_{\varphi }$ belongs to the class
$2$-$C$-normal.

\begin{proof}
$\left( i\right) $ According to \cite[Proposition 5.1]{PT}, $S_{\varphi }$
belongs to the class $C$-normal, therefore $S_{\varphi }$ is $2$-$C$-normal.

$(ii)$ Note that $S_{\varphi }^{\ast }=S_{\overline{\varphi }}$, then
\begin{eqnarray*}
CS_{\varphi }S_{\varphi }^{\ast }CS_{\varphi }^{\ast }(f(x)) &=&CS_{\varphi
}S_{\overline{\varphi }}CT_{\overline{\varphi }}(f(x)) \\
&=&CS_{\varphi }S_{\overline{\varphi }}C(\overline{\varphi (x)}f(x)) \\
\ &=&CS_{\varphi }S_{\overline{\varphi }}(\varphi (-x)\overline{f(-x)}) \\
&=&C\ \varphi (x)\overline{\varphi (x)}\varphi (-x)\overline{f(-x)} \\
&=&\overline{\varphi (-x)}\varphi (-x)\overline{\varphi (x)}f(x) \\
&=&\left\vert \varphi (-x)\right\vert ^{2}\overline{\varphi (x)}f(x)\text{.}
\end{eqnarray*}%
\ \ \ \ \ \ \ \ \ \ \ \ \ \ \ \ \ \ \ \ \ \ \ \ \ \

From another perspective, we obtain
\begin{eqnarray*}
S_{\varphi }^{\ast }CS_{\varphi }^{\ast }S_{\varphi }C(f(x)) &=&S_{\overline{%
\varphi }}CS_{\overline{\varphi }}S_{\varphi }\overline{f(-x)} \\
&=&S_{\overline{\varphi }}C\text{ }\overline{\varphi (x)}\varphi (x)%
\overline{f(-x)} \\
\ \ &=&S_{\overline{\varphi }}\varphi (-x)\text{ }\overline{\varphi (-x)}f(x)
\\
&=&\overline{\varphi (x)}\left\vert \varphi (-x)\right\vert ^{2}f(x) \\
&=&\left\vert \varphi (-x)\right\vert ^{2}\overline{\varphi (x)}f(x)\text{.}
\end{eqnarray*}

$\ \ \ \ \ \ \ \ \ \ \ \ \ \ \ \ \ \ \ \ \ \ \ \ \ \ \ $

Therefore, we obtain
\begin{equation*}
CS_{\varphi }S_{\varphi }^{\ast }CS_{\varphi }^{\ast }=S_{\varphi }^{\ast
}CS_{\varphi }^{\ast }S_{\varphi }C\text{.}
\end{equation*}%
\ This shows that $S_{\varphi }$ is $2$-$C$-normal operator.
\end{proof}

\begin{remark}
We know that from \cite{PT} any normal operator $S\in \mathcal{B}(\mathcal{K}%
)$ is unitarily equivalent to the multiplication operator $M_{\varphi }$,
meaning $M_{\varphi }=UTU^{\ast }$, for some unitary $U\in \mathcal{B}(%
\mathcal{K},L^{2}(X,%
{\mu}%
))$. Suppose $(UCU^{\ast })f(x)=\overline{f(x)}$. Consequently, $T$ is $C$%
-normal, so $T$ is $2$-$C$-normal.
\end{remark}

\begin{theorem}
\label{TH6}Suppose $S$ $\in $ $\mathcal{B}(\mathcal{K})$ is an invertible.
Whenever, $S^{\ast }(C\left\vert S^{\ast }\right\vert C)=(C\left\vert
S^{\ast }\right\vert C)S^{\ast }$, one has that $S$ is normal if and only if
it is $2$-$C$-normal.
\end{theorem}

\begin{proof}
Taking $S$ to be normal, we have%
\begin{equation*}
S^{\ast }(C\left\vert S^{\ast }\right\vert C)=(C\left\vert S^{\ast
}\right\vert C)S^{\ast }
\end{equation*}%
\begin{eqnarray*}
&\Longrightarrow &S^{\ast }(C\left\vert S^{\ast }\right\vert
C)^{2}=(C\left\vert S^{\ast }\right\vert C)^{2}S^{\ast } \\
&\Longrightarrow &S^{\ast }C\left\vert S^{\ast }\right\vert
^{2}C=C\left\vert S^{\ast }\right\vert ^{2}CS^{\ast } \\
&\Longrightarrow &S^{\ast }CSS^{\ast }C=CSS^{\ast }CS^{\ast } \\
&\Longrightarrow &S^{\ast }CS^{\ast }SC=CSS^{\ast }CS^{\ast }\text{.}
\end{eqnarray*}

Therefore, $S$ is $2$-$C$-normal.

Assume that $S$ is a $2$-$C$-normal. Hence
\begin{equation*}
S^{\ast }CS^{\ast }SC=CSS^{\ast }CS^{\ast }\
\end{equation*}%
\begin{eqnarray*}
&\Longrightarrow &S^{\ast }C\left\vert S\right\vert ^{2}C=C\left\vert
S^{\ast }\right\vert ^{2}CS^{\ast } \\
\ \ &\Longrightarrow &S^{\ast }C\left\vert S\right\vert ^{2}C=S^{\ast
}C\left\vert S^{\ast }\right\vert ^{2}C \\
\ &\Longrightarrow &(S^{\ast })^{-1}(S^{\ast }C\left\vert S\right\vert
^{2}C)=(S^{\ast })^{-1}(S^{\ast }C\left\vert S^{\ast }\right\vert ^{2}C) \\
&\Longrightarrow &C\left\vert S\right\vert ^{2}C=C\left\vert S^{\ast
}\right\vert ^{2}C \\
\ &\Longrightarrow &\left\vert S\right\vert ^{2}=\left\vert S^{\ast
}\right\vert ^{2} \\
&\Longrightarrow &S^{\ast }S=SS^{\ast }\text{.}
\end{eqnarray*}

Therefore, $S$ is normal.
\end{proof}

\begin{theorem}
\label{TH7}\ For an invertible operator $S$ $\in \mathcal{B}(\mathcal{K})$
which is $2$-$C$-normal . The following three properties are equivalent

$(1)$ $S^{\ast }$ is an isometric.

$(2)$ $S$ is an isometric.

$(3)$ $S$ is unitary.
\end{theorem}

\begin{proof}
$(1)\Longrightarrow (2)$ Assume that $S^{\ast }$ is an isometry, then $%
SS^{\ast }=I$. Since $S$ belongs to the class $2$-$C$-normal, then
\begin{equation*}
S^{\ast }CS^{\ast }SC=CSS^{\ast }CS^{\ast }\text{.}
\end{equation*}%
So,
\begin{eqnarray*}
\left. S^{\ast }CS^{\ast }SC=CICS^{\ast }\right. &\Longrightarrow &S^{\ast
}CS^{\ast }SC=S^{\ast } \\
&\Longrightarrow &S^{\ast }(CS^{\ast }SC-I)=0 \\
&\Longrightarrow &(S^{\ast })^{-1}S^{\ast }(CS^{\ast }SC-I)=0 \\
&\Longrightarrow &CS^{\ast }SC=I \\
&\Longrightarrow &CCS^{\ast }SCC=CIC \\
&\Longrightarrow &S^{\ast }S=I\text{.}
\end{eqnarray*}%
This implies that $S$ is an isometry.

$(2)\Longrightarrow (3)$ Assume that $S$ is an isometry then, $S^{\ast }S=I$
. We have $S$ is $2$-$C$-normal. So
\begin{eqnarray*}
\left. S^{\ast }CS^{\ast }SC=CSS^{\ast }CT^{\ast }\right. &\Longrightarrow
&S^{\ast }CIC=CSS^{\ast }CS^{\ast } \\
&\Longrightarrow &S^{\ast }=CSS^{\ast }CS^{\ast } \\
&\Longrightarrow &(CSS^{\ast }C-I)S^{\ast }=0 \\
&\Longrightarrow &(CSS^{\ast }C-I)S^{\ast }(S^{\ast })^{-1}=0 \\
&\Longrightarrow &CSS^{\ast }C=I \\
&\Longrightarrow &CCSS^{\ast }CC=CIC \\
&\Longrightarrow &SS^{\ast }=I\text{.}
\end{eqnarray*}%
This gives $S^{\ast }S=SS^{\ast }=I$. Hence $S$ is a unitary.

$(3)\Longrightarrow (1)$ If $S$ is a unitary operator then $S^{\ast
}S=SS^{\ast }=I$ . So, $S^{\ast }$ is an isometry.
\end{proof}

\begin{theorem}
\label{TH8}Let $T,S$ $\in \ \mathcal{B}(\mathcal{K})$ such that $T=CSC$. Then

$(a)$ $T$ is $C$-normal $\Leftrightarrow S$ is $C$-normal.

$(b)$ $T$ is $2$-$C$-normal $\Leftrightarrow S$ is $2$-$C$-normal.
\end{theorem}

\begin{proof}
$(a)$ Assume that $T=CSC$. If $T$ is $C$-normal, then%
\begin{equation*}
CT^{\ast }TC=TT^{\ast }
\end{equation*}%
\begin{eqnarray*}
&\Leftrightarrow &C(CS^{\ast }C)(CSC)C=(CSC)(CS^{\ast }C) \\
&\Leftrightarrow &S^{\ast }S=CSS^{\ast }C\text{.}
\end{eqnarray*}

Therefore, $S$ belongs to the class $C$-normal.\

$(b)$ Suppose $T$ belongs to the class $2$-$C$-normal. Then%
\begin{equation*}
T^{\ast }CT^{\ast }TC=CTT^{\ast }CT^{\ast }
\end{equation*}%
\begin{eqnarray*}
&\Longleftrightarrow &(CS^{\ast }C)C(CS^{\ast }C)(CSC)C=C(CSC)(CS^{\ast
}C)C(CS^{\ast }C) \\
&\Longleftrightarrow &CS^{\ast }CS^{\ast }S=SS^{\ast }CS^{\ast }C \\
&\Longleftrightarrow &S^{\ast }CS^{\ast }SC=CSS^{\ast }CS^{\ast }\text{.}
\end{eqnarray*}

Therefore, $S$ is $2$-$C$-normal.
\end{proof}

\begin{corollary}
Let $S$ $\in $ $\mathcal{B}(\mathcal{K})$. Then
\end{corollary}

$(i)$ Whenever $S$ is $C$-normal, the operator $CSC$ is also $C$-normal.

$(ii)$ Whenever $S$ is $2$-$C$-normal, the operator $CSC$ is also $2$-$C$%
-normal.

Denote by $N_{2.C}(\mathcal{K})$ the family of $2$-$C$-normal operators on $%
\mathcal{K}$.

\begin{proposition}
The class $N_{2.C}(\mathcal{K})$ is a norm closed in $\mathcal{B}(\mathcal{K}%
)$.

\begin{proof}
If $S\in \overline{N_{2.C}(\mathcal{K})}$, then one can find a sequence $%
\{S_{n}\}_{n}$ $\in $ $N_{2.C}(\mathcal{K})$ with
\begin{equation*}
\lim\limits_{n\rightarrow \infty }\Vert S_{n}-S\Vert =0\text{.}
\end{equation*}%
Therefore, we have
\begin{eqnarray*}
\left\Vert S^{\ast }CS^{\ast }SC-CSS^{\ast }CS^{\ast }\right\Vert &\leq
&\left\Vert S^{\ast }CS^{\ast }SC-S_{n}^{\ast }CS^{\ast }SC\right\Vert
+\left\Vert S_{n}^{\ast }CS^{\ast }SC-S_{n}^{\ast }CS_{n}^{\ast
}SC\right\Vert \\
&&+\left\Vert S_{n}^{\ast }CS_{n}^{\ast }SC-S_{n}^{\ast }CS_{n}^{\ast
}S_{n}C\right\Vert +\left\Vert S_{n}^{\ast }CS_{n}^{\ast
}S_{n}C-CS_{n}S_{n}^{\ast }CS_{n}^{\ast }\right\Vert \\
&&+\left\Vert CS_{n}S_{n}^{\ast }CS_{n}^{\ast }-CSS_{n}^{\ast }CS_{n}^{\ast
}\right\Vert +\left\Vert CSS_{n}^{\ast }CS_{n}^{\ast }-CSS^{\ast
}CS_{n}^{\ast }\right\Vert \\
&&+\left\Vert CSS^{\ast }CS_{n}^{\ast }-CSS^{\ast }CS^{\ast }\right\Vert \\
&\leq &\left\Vert (S^{\ast }-S_{n}^{\ast })CS^{\ast }SC\right\Vert
+\left\Vert S_{n}^{\ast }C(S^{\ast }-S_{n}^{\ast })SC\right\Vert \\
&&+\left\Vert S_{n}^{\ast }CS_{n}^{\ast }(S-S_{n})C\right\Vert +\left\Vert
S_{n}^{\ast }CS_{n}^{\ast }S_{n}C-CS_{n}S_{n}^{\ast }CS_{n}^{\ast
}\right\Vert \\
&&+\left\Vert C(S_{n}-S)S_{n}^{\ast }CS_{n}^{\ast }\right\Vert +\left\Vert
CS(S_{n}^{\ast }-S^{\ast })CS_{n}^{\ast }\right\Vert \\
&&+\left\Vert CSS^{\ast }C(S_{n}^{\ast }-S^{\ast })\right\Vert
\longrightarrow 0\text{.}
\end{eqnarray*}%
So,
\begin{equation*}
\left\Vert S^{\ast }CS^{\ast }SC-CSS^{\ast }CS^{\ast }\right\Vert =0\text{.}
\end{equation*}%
\ Therefore,
\begin{equation*}
S^{\ast }CS^{\ast }SC=CSS^{\ast }CS^{\ast }\text{.}
\end{equation*}%
Thus, $S\in N_{2.C}(\mathcal{K})$. Hence $N_{2.C}(\mathcal{K})$ is a norm
closed in $\mathcal{B}(\mathcal{K})$.
\end{proof}
\end{proposition}

\begin{theorem}
\label{TH9}For $S$ $\in $ $\mathcal{B}(\mathcal{K})$, assume that it is $2$-$%
C$-normal. Then

$(i)$ Suppose $S$ is injective. Then, $\ker (CS^{\ast }CS^{\ast })=\ker
(S^{\ast }CS^{\ast }SC)$.

$(ii)$ Suppose $S^{\ast }$ is injective. Then, $\ker (CSC)=\ker (CSS^{\ast
}CS^{\ast })$.
\end{theorem}

\begin{proof}
Assume that $S$ is a $2$-$C$-normal operator.

$(i)$ Let $x\in \ker (CS^{\ast }CS^{\ast })$, thus $CS^{\ast }CS^{\ast
}(x)=0 $, this implies that $S^{\ast }CS^{\ast }(x)=0$. Since $S$ is an
injective, it follows that $SS^{\ast }CS^{\ast }(x)=0$ which implies that $%
CSS^{\ast }CS^{\ast }(x)=0$. As $S$ is a $2$-$C$-normal then $S^{\ast
}CS^{\ast }SC(x)=0$, so $x\in \ker (S^{\ast }CS^{\ast }SC)$.

This implies that
\begin{equation*}
\ker (CS^{\ast }CS^{\ast })\subseteq \ker (S^{\ast }CS^{\ast }SC)\text{.}
\end{equation*}

Now, let $x\in \ker (S^{\ast }CS^{\ast }SC)$ then $S^{\ast }CS^{\ast
}SC(x)=0 $. Since $S$ is $2$-$C$-normal, so $CSS^{\ast }CS^{\ast }(x)=0$,
then $SS^{\ast }CS^{\ast }(x)=0$. Given that $S$ is injective one has $%
S^{\ast }CS^{\ast }(x)=0$, this implies that $CS^{\ast }CS^{\ast }(x)=0$
hence $x\in \ker (CS^{\ast }CS^{\ast })$. Thus
\begin{equation*}
\ker (S^{\ast }CS^{\ast }SC)\subseteq \ker (CS^{\ast }CS^{\ast })\text{.}
\end{equation*}

Then, $\ker (CS^{\ast }CS^{\ast })=\ker (S^{\ast }CS^{\ast }SC)$.

$(ii)$ Let $x\in \ker (CSC)$, then
\begin{eqnarray*}
\left. CSC(x)=0\right. &\Longrightarrow &SC(x)=0 \\
&\Longrightarrow &S^{\ast }SC(x)=0 \\
&\Longrightarrow &S^{\ast }CS^{\ast }SC(x)=0 \\
&\Longrightarrow &CSS^{\ast }CS^{\ast }(x)=0\text{.}
\end{eqnarray*}%
This means that $x\in \ker (CSS^{\ast }CS^{\ast })$. This implies $\ker
(CSC)\subseteq \ker (CSS^{\ast }CS^{\ast })$.

For every $x\in \ker (CSS^{\ast }CS^{\ast })$, since $S^{\ast }$is an
injective operator. It follows that
\begin{eqnarray*}
CSS^{\ast }CS^{\ast }(x) &=&0\  \\
&\Longrightarrow &S^{\ast }CS^{\ast }SC(x)=0 \\
&\Longrightarrow &CS^{\ast }SC(x)=0 \\
&\Longrightarrow &S^{\ast }SC(x)=0 \\
&\Longrightarrow &SC(x)=0 \\
&\Longrightarrow &CSC(x)=0
\end{eqnarray*}%
Thus $x$ belongs to $\ker (CSC)$. This implies that
\begin{equation*}
\ker (CSS^{\ast }CS^{\ast })\subseteq \ker (CSC)\text{.}
\end{equation*}%
Therefore,
\begin{equation*}
\ker (CSC)=\ker (SS^{\ast }CS^{\ast })\text{.}
\end{equation*}
\end{proof}

\begin{proposition}
Take $S\in \mathcal{B}(\mathcal{K})$ to be $2$-$C$-normal. Assume that $%
SS^{\ast }CS^{\ast }C$ or $CS^{\ast }CS^{\ast }S$ is bounded below. Then $S$
must be bounded below.
\end{proposition}

\begin{proof}
By assuming $S$ fails to be bounded below, the point $0$ necessarily belongs
to $\sigma _{ap}(S)$. One can construct a sequence $(x_{n})_{n}$, where $%
\left\Vert x_{n}\right\Vert =1$ and $\left\Vert Sx_{n}\right\Vert
\longrightarrow 0$. So, we have
\begin{equation*}
\left\Vert CS^{\ast }CS^{\ast }Sx_{n}\right\Vert \leq \left\Vert CS^{\ast
}CS^{\ast }\right\Vert \left\Vert Sx_{n}\right\Vert \longrightarrow 0\text{,
}
\end{equation*}%
and
\begin{eqnarray*}
\left\Vert SS^{\ast }CS^{\ast }Cx_{n}\right\Vert &=&\left\Vert CS^{\ast
}CS^{\ast }SCCx_{n}\right\Vert \\
&=&\left\Vert CS^{\ast }CS^{\ast }Sx_{n}\right\Vert \\
&\leq &\left\Vert CS^{\ast }CS^{\ast }\right\Vert \left\Vert
Sx_{n}\right\Vert \longrightarrow 0\text{.}
\end{eqnarray*}

Since $\left\Vert Cx_{n}\right\Vert =\left\Vert x_{n}\right\Vert =1$, then
we get $0\in \sigma _{ap}(CS^{\ast }CS^{\ast }S)$ and $0\in \sigma
_{ap}(SS^{\ast }CS^{\ast }C)$.

Therefore, $CS^{\ast }CS^{\ast }S$ and $SS^{\ast }CS^{\ast }C$ are not
bounded below.
\end{proof}

\begin{proposition}
Suppose $S\in \mathcal{B}(\mathcal{K})$ is invertible and $2$-$C$-normal.
Then, $S$ is bounded below precisely when $S^{\ast }$ also a bounded below.
\end{proposition}

\begin{proof}
\bigskip Take $S$ \ to be a $2$-$C$-normal does not satisfy the
bounded-below property, one obtains $0\in \sigma _{ap}(S)$. It is possible
to find $(x_{n})_{n}$, where $\left\Vert x_{n}\right\Vert =1$ with $%
\left\Vert Sx_{n}\right\Vert \longrightarrow 0$. We have%
\begin{eqnarray*}
\left\Vert CS^{\ast }Cx_{n}\right\Vert &=&\left\Vert (S^{\ast
})^{-1}S^{-1}CCSS^{\ast }CS^{\ast }Cx_{n}\right\Vert \\
&=&\left\Vert (SS^{\ast })^{-1}CS^{\ast }CS^{\ast }Sx_{n}\right\Vert \\
&\leq &\left\Vert (SS^{\ast })^{-1}CS^{\ast }CS^{\ast }\right\Vert
\left\Vert Sx_{n}\right\Vert \longrightarrow 0\text{.}
\end{eqnarray*}%
Since $\left\Vert Cx_{n}\right\Vert =\left\Vert x_{n}\right\Vert =1$, then $%
0\in \sigma _{ap}(CS^{\ast }C)=\sigma _{ap}(S^{\ast })$ which that $S^{\ast
} $does not satisfy the bounded-below property.

Assume that $S^{\ast }$ is not bounded below. Then belongs to the
approximate point spectrum of $S^{\ast }$, so one can find a sequence $%
(x_{n})_{n}$ with $\left\Vert x_{n}\right\Vert =1$ such that $\left\Vert
S^{\ast }x_{n}\right\Vert \longrightarrow 0$. We have
\begin{eqnarray*}
\Vert CSCx_{n}\Vert &=&\Vert C(S^{\ast })^{-1}C(S^{\ast })^{-1}S^{\ast
}CS^{\ast }SCx_{n}\Vert \\
&=&\Vert C(S^{\ast })^{-1}C(S^{\ast })^{-1}CSS^{\ast }CS^{\ast }x_{n}\Vert \\
&\leq &\Vert C(S^{\ast })^{-1}C(S^{\ast })^{-1}CSS^{\ast }C\Vert \Vert
S^{\ast }x_{n}\Vert \longrightarrow 0\text{.}
\end{eqnarray*}%
Since $\Vert Cx_{n}\Vert =\Vert x_{n}\Vert =1$, then $0\in \sigma
_{ap}(CSC)=\sigma _{ap}(S)$ which that$\ S$ is not bounded below.
\end{proof}

\textbf{Data availability}

No data was used for the research described in the article.

\end{document}